\documentclass[11pt,a4paper,leqno]{amsart}

\usepackage{amssymb,enumerate,graphicx}
\usepackage[initials]{amsrefs}

\newcommand{\UP}{\blacktriangle}                       
\newcommand{\DOWN}{\blacktriangledown}                        

\numberwithin{equation}{section}
          
\theoremstyle{plain}
\newtheorem{theorem}{Theorem}[section]
\newtheorem{proposition}[theorem]{Proposition}
\newtheorem{lemma}[theorem]{Lemma}
\newtheorem{corollary}[theorem]{Corollary}

\theoremstyle{definition}

\newtheorem{example}[theorem]{Example}
\newtheorem{remark}[theorem]{Remark}

\begin{document}
 
\title{Tolerances induced by irredundant coverings}

\author[J.~J{\"a}rvinen]{Jouni J{\"a}rvinen}
\address{J.~J{\"a}rvinen, Sirkankuja 1, 20810~Turku, Finland}
\email{Jouni.Kalervo.Jarvinen@gmail.com}
\urladdr{\url{http://sites.google.com/site/jounikalervojarvinen/}}

\author[S.~Radeleczki]{S{\'a}ndor Radeleczki}
\thanks{Acknowledgement:  The research of the second author was carried out as
part of the TAMOP-4.2.1.B-10/2/KONV-2010-0001 project supported by the
European Union, co-financed by the European Social Fund.}
\address{S.~Radeleczki, Institute of Mathematics\\ 
University of Miskolc\\3515~Miskolc-Egyetemv{\'a}ros\\Hungary}
\email{matradi@uni-miskolc.hu}
\urladdr{\url{http://www.uni-miskolc.hu/~matradi/}}

\begin{abstract}
In this paper, we consider tolerances induced by irredundant coverings. Each tolerance $R$ on $U$ determines a quasiorder $\lesssim_R$ by setting
$x \lesssim_R y$ if and only if $R(x) \subseteq R(y)$.  We prove that for a tolerance $R$ induced by a covering $\mathcal{H}$ of $U$, 
the covering $\mathcal{H}$ is irredundant if and only if the quasiordered set $(U, \lesssim_R)$ is bounded by minimal elements 
and the tolerance $R$ coincides with the product ${\gtrsim_R} \circ {\lesssim_R}$. We also show that in such a case 
$\mathcal{H} = \{ {\uparrow}m \mid \text{$m$ is minimal in $(U,\lesssim_R)$} \}$, and for each minimal $m$, we have $R(m) = {\uparrow} m$. 
Additionally, this irredundant covering $\mathcal{H}$ inducing $R$ consists of some blocks of the tolerance $R$. 
We give necessary and sufficient conditions under which $\mathcal{H}$ and the set of $R$-blocks coincide. These results are established by applying the 
notion of Helly numbers of quasiordered sets.
\end{abstract}

\maketitle

\section{Introduction and motivation} \label{Sec:Introduction}

In this paper, we consider tolerances (reflexive and symmetric binary relations) and, in particular, 
tolerances induced by irredundant coverings. The term \emph{tolerance relation} was introduced
in the context of visual perception theory by E.~C.~Zeeman \cite{Zeeman62}, 
motivated by the fact that indistinguishability of ``points'' in the visual 
world is limited by the discreteness of retinal receptors. 

Let $R$ be a tolerance on $U$ and suppose that $R$ is interpreted as a similarity relation such that $x \, R \, y$
means that $x$ and $y$ are similar in terms of the knowledge we have about them. Of course, 
each object is similar to itself, and if $x$ is similar to $y$, it is natural to assume that
$y$ is similar to $x$. Similarity relations are not necessarily transitive: we may have a finite sequence of objects 
$x_1, x_2, \ldots, x_n$ such that each  two consecutive objects $x_i$ and $x_{i+1}$ are similar, but $x_1$ and
$x_n$ are very different from each other.

Next, we recall some elementary notions and facts about rough set theory and formal concept analysis
from \cite{Jarv07} and \cite{ganter1999formal}, respectively.
Let $R$ be a tolerance on $U$. The \emph{upper $R$-approximation} of a set $X\subseteq U$ is 
\begin{align*}
 & X^\UP = \{x \in U \mid R(x) \cap X \neq \emptyset\} \\
\intertext{and the \emph{lower $R$-approximation} of $X$ is}
& X^\DOWN = \{x \in U \mid R(x) \subseteq X\}.
\end{align*}
Here  $R(x) = \{ y \in U \mid x \, R \, y$\} denotes the \emph{$R$-neighbourhood} of the element $x$.
Then, $X^\UP$ can be viewed as the set of elements that possible are in $X$, because in $X$ is at least
one element similar to them. The set $X^\DOWN$ can be considered as the set of elements certainly in $X$,
because for these elements, all elements similar to them are in $X$. 

Let us denote $\wp(U)^\DOWN = \{X^\DOWN \mid X \subseteq U\}$ and $\wp(U)^\UP = \{X^\UP \mid X\subseteq U\}$.
The complete lattices $(\wp(U)^\UP,\subseteq)$ and $(\wp(U)^\DOWN,\subseteq)$ are ortholattices. 
In $\wp(U)^\UP$, the orthocomplementation is $^\bot \colon X^\UP \mapsto X^{\UP c \UP}$,
and the map $^\top \colon X^\DOWN \mapsto X^{\DOWN c \DOWN}$ is the orthocomplementation operation of $\wp(U)^\DOWN$,
where for any $A \subseteq U$, $A^c = U \setminus A$ is the set-theoretical complement of $A$.

A \emph{formal context} $\mathbb{K} = (G,M,I)$ consists of two
sets $G$ and $M$, and a relation $I$ from $G$ to $M$. The elements of $G$ are 
called the \emph{objects} and the elements of $M$ are called \emph{attributes} of
the context $\mathbb{K}$. For $A \subseteq G$ and $B \subseteq M$, we define
$A' = \{ m \in M \mid g \, I \, m \text{ for all } g \in A\}$ and 
$B' = \{ g \in G \mid g \, I \, m \text{ for all } m \in B\}$.
A \emph{formal concept} of the context $(G,M,I)$ is a pair $(A,B)$ with
$A \subseteq G$, $B \subseteq M$, $A' = B$, and $B' = A$. 
We call $A$ the \emph{extent} and $B$ the \emph{intent} of the concept $(A,B)$.
The set of all concepts of the context $\mathbb{K} = (G,M,I)$ is denoted 
by $\mathfrak{B}(\mathbb{K})$. 
The set $\mathfrak{B}(\mathbb{K})$ is ordered by
\[
 (A_1,B_1) \leq (A_2,B_2) \iff A_1 \subseteq A_2 \iff B_1 \supseteq B_2.
\]
With respect to this order, $\mathfrak{B}(\mathbb{K})$ forms a complete lattice,
called the \emph{concept lattice} of the context $\mathbb{K}$.
If $I\subseteq U\times U$ is an irreflexive and symmetric relation, then
the concept lattice $\mathfrak{B}(\mathbb{K})$ corresponding to the
context $\mathbb{K}=(U,U,I)$ is a complete ortholattice. 

If $R$ is a tolerance relation on $U$,
then its complement $R^c = U \times U \setminus R$ is irreflexive and symmetric. 
Thus, the concept lattice $\mathfrak{B}(\mathbb{K})$ of the context $\mathbb{K} = (U,U,R^c)$ is a complete 
ortholattice. We proved in \cite{JarvRade2014} that 
\[ 
  \mathfrak{B}(\mathbb{K}) = \{ (A,A^\top) \mid A \in \wp(U)^\DOWN \},
\]
where $^\top$ is the orthocomplement defined in $\wp(U)^\DOWN$. Additionally,
for $(A,A^\top) \in \mathfrak{B}(\mathbb{K})$, its orthocomplement is $(A^\top, A)$. As noted in
\cite{JarvRade2014}, for any tolerance $R$, the complete ortholattices
$\wp(U)^\DOWN$, $\wp(U)^\UP$, and $\mathfrak{B}(\mathbb{K})$ are isomorphic. 

Let $\mathcal{L} = (L,\leq)$ be a lattice with a least element $0$. 
The lattice $\mathcal{L}$ is \emph{atomistic}, if any element of $L$ is the join of atoms below it. 
We proved in \cite{JarvRade2014} that for any tolerance $R$ on $U$, the following conditions are equivalent:
\begin{enumerate}[({C}1)]
 \item $R$ is induced by an irredundant covering;
 \item $\wp(U)^\DOWN$ and $\wp(U)^\UP$ are atomistic complete Boolean lattices;
 \item $\mathfrak{B}({\mathbb{K}})$ is an atomistic complete Boolean lattice.
\end{enumerate} 
This means that knowing when $R$ is induced by an irredundant covering is important, because this fully
characterises the case when the lattices $\wp(U)^\DOWN$, $\wp(U)^\UP$, and $\mathfrak{B}({\mathbb{K}})$ are atomistic 
complete Boolean lattices. Note also that if $R$ is a tolerance on $U$ induced by an irredundant covering, then
the set of atoms of $\wp(U)^\UP$ is $\{R(x) \mid \text{$R(x)$ is a block}\}$.
Because the lattice-join in $\wp(U)^\UP$ is the set-theoretical union, each element of $\wp(U)^\UP$ can be
given as a union of some $R$-neighbourhoods forming a block. Notice that the basic notions, such as coverings and blocks, 
related to tolerances are formally defined in Section~\ref{Sec:Preliminaries}.

In this paper, we continue the study of tolerances induced by an irredundant covering.
Our current study is done in the setting of quasiordered sets. Clearly, each tolerance $R$ on $U$ defines a quasiorder by
$x \lesssim_R y$ if and only if $R(x) \subseteq R(y)$. We show that for a tolerance $R$ induced by a covering $\mathcal{H}$ of $U$, 
the covering $\mathcal{H}$ is irredundant if and only if the quasiordered set $(U, \lesssim_R)$ is bounded by minimal elements 
and the tolerance $R$ coincides with the product ${\gtrsim_R} \circ {\lesssim_R}$. We also characterise the case when 
this $\mathcal{H}$ and the set of $R$-blocks coincide. 

This paper is structured as follows. Section~\ref{Sec:Preliminaries} recalls shortly the basic facts about tolerances.
In Section~\ref{Sec:Quasiorders}, we consider the relationship between tolerances and quasiorders, and we
characterise the case in which a tolerance is induced by an irredundant covering in terms of minimal elements
of a quasiordered set. We adopt from combinatorics the notion of Helly numbers in Section~\ref{Sec:Helly}, and we 
connect Helly numbers of quasiordered sets and blocks of a tolerance forming an irredundant covering. In particular, by using this result
for a tolerance $R$ induced by an irredundant covering $\mathcal{H}$, we give necessary and sufficient conditions
for $\mathcal{H} = \mathcal{B}(R)$, that is, for normality of the covering $\mathcal{H}$. The section ends with a consideration
of tolerances induced by finite distributive lattices.

\section{Preliminaries} \label{Sec:Preliminaries}
 
Let us first recall some notions and basic results from  \cite{Bartol2004,Schreider75}.
Let $R$ be a tolerance on $U$, that is, $R$ is a reflexive and symmetric binary relation on $U$. A nonempty subset
$X$ of $U$ is an \emph{$R$-preblock} if $X^2 \subseteq R$. An \emph{$R$-block} is a maximal $R$-preblock,
that is, an $R$-preblock $B$ is an $R$-block if $B \subseteq X$ implies $B = X$ for any $R$-preblock $X$. 
Thus, any subset $\emptyset \neq X \subseteq U$ is an $R$-preblock if and only if it is contained in some $R$-block.
The family of all $R$-blocks is denoted by $\mathcal{B}(R)$. If there is no danger of confusion, we simply call $R$-blocks as blocks. 
Each tolerance $R$ is completely determined by its blocks, that is, $a \, R \, b$ if and only if there exists a block 
$B \in \mathcal{B}(R)$ such that $a,b \in B$. This can be written in the form 
$R = \bigcup \{ B^2 \mid B \in \mathcal{B}(R)\}$.
It is also known that $B$ is a block if and only if 
\[  B = \{x \in U \mid B \subseteq R(x)\}. \]

We recall from \cite{Chajda91, Pogonowski81} the following lemma combining some conditions for an $R$-neighbourhood forming a block.

\begin{lemma} \label{Lem:ToleranceBlocks}
Let $R$ be a tolerance on $U$ and $x \in U$. Then, the following are equivalent:
\begin{enumerate}[\rm (a)]
 \item $R(x)$ is a block;
 \item $a \, R \, b$ for all $a,b \in R(x)$;
 \item $a \in R(x)$ implies $R(x) \subseteq R(a)$;
 \item $R(x) = \bigcap \{ R(a) \mid a \in R(x) \}$.
\end{enumerate}
\end{lemma}

A collection $\mathcal{H}$ of nonempty subsets of $U$ is called a \emph{covering} of $U$ if $\bigcup \mathcal{H} = U$.
A covering $\mathcal{H}$ is \emph{irredundant} if $\mathcal{H} \setminus \{X\}$ is not a covering for any $X \in \mathcal{H}$.
Clearly, the blocks $\mathcal{B}(R)$ of a tolerance $R$ form a covering, but this covering is not generally 
irredundant. On the other hand, each covering $\mathcal{H}$ of $U$ defines a tolerance 
$R_\mathcal{H} = \bigcup \{ X^2 \mid X \in \mathcal{H}\}$, called the \emph{tolerance induced} by $\mathcal{H}$.
Obviously, the sets in $\mathcal{H}$ are preblocks of $R_\mathcal{H}$.

Let $R$ be a tolerance induced by a covering $\mathcal{H}$ of $U$. In \cite[Proposition~3.7]{JarvRade2014}, 
we showed the equivalence of the following statements:
\begin{enumerate}[({D}1)]
 \item $\mathcal{H}$ an irredundant covering of $U$;
 \item $\mathcal{H} \subseteq \{ R(x) \mid x \in U \}$.
\end{enumerate}

\begin{example} \label{Exa:Schreider}
Let us consider the following setting modified from an example appearing in \cite{Schreider75}.
Let $\mathbb{N} = \{1,2,3,\ldots \}$. We define a tolerance $R$ on $U = \wp(\mathbb{N}) \setminus \{\emptyset\}$ 
by setting for any nonempty subsets $B,C\subseteq \mathbb{N}$:
\[
(B, C) \in R \iff B \cap C \neq \emptyset.
\]
Let $n \in \mathbb{N}$ and define $K_{n} = \{B \in U \mid n \in B\}$.
Clearly, $R(\{n\}) = K_{n}$. Each $R(\{n\}) = K_{n}$
is also a tolerance block, because all sets in $K_n$ contain $n$, and hence
they are $R$-related (cf.~Lemma~\ref{Lem:ToleranceBlocks}).

Let $\mathcal{H} = \{K_n \mid n \in \mathbb{N} \}$. Then,
$\bigcup \mathcal{H} = \wp(\mathbb{N} ) \setminus \{\emptyset\}$, that is, $\mathcal{H}$ is a covering of $U$. 
The tolerance $R_\mathcal{H}$ induced by $\mathcal{H}$ is $R$, because if $(B,C) \in R$, then there exists
$n \in B \cap C$ and so $B,C \in K_n$. Conversely, if $B,C \in K_n$,
then $(B,C) \in R$, because $K_n$ is a block. Since each set in $\mathcal{H}$ is of the from $K_n = R(\{n\})$, 
$n \in \mathbb{N}$, the covering $\mathcal{H}$ is irredundant by the equivalence of (D1) and (D2).
\end{example}

\begin{lemma} \label{Lem:IrredundantBlocks}
If $R$ is a tolerance induced by an irredundant covering $\mathcal{H}$, then $\mathcal{H} \subseteq \mathcal{B}(R)$.
\end{lemma}

\begin{proof} 
Because conditions (D1) and (D2)  are equivalent, 
there exists $Y \subseteq U$ such that $\mathcal{H} = \{ R(y) \mid y \in Y \}$. Since $R$ is
induced by $\mathcal{H}$, then for each $y \in Y$, we have $a \, R \, b$ for any $a,b \in R(y)$. 
Thus, by Lemma~\ref{Lem:ToleranceBlocks}, $R(y)$ is a block for any $y$ in $Y$, and 
$\mathcal{H} \subseteq \mathcal{B}(R)$.
\end{proof}

Let $R$ be a tolerance induced by an irredundant covering $\mathcal{H}$. The inclusion $\mathcal{H} \subseteq \mathcal{B}(R)$
presented in Lemma~\ref{Lem:IrredundantBlocks} can be proper (see Example~\ref{Exa:Helly}). In Section~\ref{Sec:Helly},
we give a characterisation for the case  $\mathcal{H} = \mathcal{B}(R)$.

\begin{example} \label{Exa:Graph1}
Tolerances can be considered as simple graphs (and vice versa). A \emph{simple graph} is an undirected graph that has no loops 
(edges connected at both ends to the same vertex) and no multiple edges.
Any tolerance $R$ on $U$ determines a graph $\mathcal{G} = (U, R)$, where $U$ is interpreted as the set of vertices and $R$ as 
the set of edges. There is a line connecting $x$ and $y$ if and only if $x \, R \, y$. Because each point is $R$-related
to itself, loops connecting a point to itself are not drawn/permitted. Note that \cite{Zelinka68} contains some elementary 
connections between tolerances and undirected graphs.

A nonempty set $X \subseteq U$ is a preblock if and only if $X$ is a clique of $\mathcal{G}$, that is, all points in $X$ are connected
by an edge of $\mathcal{G}$. A block of $R$ is thus a maximal clique. For any $x \in U$, the neighbourhood $R(x)$ is the set of points connected
to $x$. If this set forms a clique, then by Lemma~\ref{Lem:ToleranceBlocks}, $R(x)$ is necessarily a block, that is, a maximal clique.
In general, listing all maximal cliques may require exponential time as there exist graphs with exponentially many maximal cliques. 
The article \cite{Alt78} presents an improved version of the fundamental Harary--Ross algorithm \cite{Harary57} enumerating all cliques of
a graph.  

Note also that blocks (or maximal cliques) have several applications. For instance, in \cite{Niemine78} are studied
tolerances on lattices and their block description, and in \cite{Niemine79} blocks of compatible
relations on error algebras are considered.
\end{example}

\section{Tolerances induced by quasiorders} \label{Sec:Quasiorders}

Let $\lesssim$ be a quasiorder on $U$, that is, $\lesssim$ is a reflexive and transitive binary relation on $U$. 
The pair $(U, \lesssim)$ is called a \emph{quasiordered set}. An element $x$ is a \emph{minimal element} of $(U,\lesssim)$, 
if $y \lesssim x$ implies $y \gtrsim x$, where $\gtrsim$ denotes the inverse relation of $\lesssim$. 
Note that this definition coincides with the definition of minimality in partially ordered sets, because of the antisymmetry of a partial 
order $\leq$. We say that $(U,\lesssim)$ is \emph{bounded by minimal elements} if for any $x \in U$, there exists a minimal 
$m \in U$ such that $m \lesssim x$. 

The relation ${\gtrsim} \circ {\lesssim}$ is a tolerance on $U$ and it is denoted simply by $\approx$. The upset ${\lesssim}(x) = \{y \mid x \lesssim y\}$ 
is written as ${\uparrow}x$. Note that for any $x \in U$, ${\uparrow} x$ is a preblock of $\approx$, because $a,b \in {\uparrow} x$ is equivalent to
$a \gtrsim x \lesssim b$, that is, $a \approx b$.

\begin{lemma}\label{Lem:MinimalElements}
Let $\lesssim$ be a quasiorder on $U$.
\begin{enumerate}[\rm (a)]
 \item If $m$ is a minimal element, then ${\approx}(m) = {\uparrow} m$ and ${\uparrow} m$ is a block of $\approx$.
 \item If $(U,\lesssim)$ is bounded by minimal elements, then $\mathcal{H} = \{ {\uparrow} m \mid \text{$m$ is minimal\,} \}$ is an
 irredundant covering of $U$, and the tolerance induced by $\mathcal{H}$ is $\approx$.
\end{enumerate}
\end{lemma}

\begin{proof}
(a) We first show that $m \approx x$ if and only if $x \in {\uparrow} m$. If $x \approx m$, then 
$x \gtrsim a \lesssim m$ for some $a \in U$. But since $m$ is minimal, $a \lesssim m$ gives 
$m \lesssim a \lesssim x$ and $x \in {\uparrow}m$. Conversely, $x,m \in {\uparrow} m$ implies $x \approx m$ 
since ${\uparrow} m$ is a preblock of $\approx$. Thus, ${\approx}(m) = {\uparrow} m$.

If $x,y \in {\uparrow}m$, then $m \lesssim x$ and $m \lesssim y$ give $x \approx y$. 
By Lemma~\ref{Lem:ToleranceBlocks} this means that ${\uparrow}m$ is a block of $\approx$.

(b) Let $x \in U$. Because $U$ is bounded by minimal elements, there exists a minimal $m$ such that
$m \lesssim x$ and $x \in {\uparrow} m$. Hence, $\bigcup \mathcal{H} = U$.

Next, we prove that the tolerance $R_\mathcal{H}$ induced by $\mathcal{H}$ equals $\approx$.
If $(x,y) \in R_\mathcal{H}$, then there exists a minimal $m$ such that $x,y \in {\uparrow} m$.
Then, $x \gtrsim m \lesssim y$ gives $x \approx y$. Conversely, $x \approx y$ means that
there exists an element $a$ such that $x \gtrsim a \lesssim y$. Thus, $a \lesssim x,y$. Since
$U$ is bounded by minimal elements, there exists a minimal $m$ such that $m \lesssim a \lesssim x,y$.
Therefore, $x,y \in {\uparrow} m$ and $(x,y) \in R_\mathcal{H}$.

Because each element of $\mathcal{H}$ is of the form $R_\mathcal{H}(m) = {\approx}(m) = {\uparrow}m$, 
$\mathcal{H}$ is irredundant, since (D1) and (D2) are equivalent.
\end{proof}

\begin{corollary}\label{Cor:Upset}
Let $(U,\lesssim)$ be bounded by minimal elements and $x,y \in U$. Then, $x \approx y$ 
if and only if there exists a minimal $m$ such that  $x,y \in {\uparrow} m$.
\end{corollary}

Let $R$ be a tolerance on $U$. It is clear that the relation $\lesssim_R$ defined by
\[ x \lesssim_R y \iff R(x) \subseteq R(y) \]
is a quasiorder on $U$. In addition, we denote by $\approx_R$ the relation ${\gtrsim_R} \circ {\lesssim_R}$.
It is easy to observe that
\[ x \approx_R y \iff (\exists a) \, R(a) \subseteq R(x) \cap R(y).\]

\begin{lemma} \label{Lem:Minimal}
If $R$ is a tolerance on $U$ induced by an irredundant covering $\mathcal{H}$, then 
$(U,\lesssim_R)$ is bounded by the minimal elements  $\{ d \in U \mid \text{$R(d)$ is a block} \}$.
\end{lemma}

\begin{proof} 
Let $R(d)$ be a block. As noted in Section~\ref{Sec:Introduction}, $R(d)$ is an atom of $\wp(U)^\UP$. 
Suppose that $x \lesssim_R d$. Then, $\emptyset \neq R(x) \subseteq R(d)$, and because $R(x) \in \wp(U)^\UP$ and 
$R(d)$ is an atom of $\wp(U)^\UP$, we have $R(d) = R(x)$. This gives $x \gtrsim_R d$, and thus $d$ is minimal. 
On the other hand, if $m$ is a minimal element, then for all $x \in U$, $x \lesssim_R m$ implies $x \gtrsim_R m$, that is, 
$R(x) \subseteq R(m)$ implies $R(x) = R(m)$. Because $R(m) \in \wp(U)^\UP$, there exists an atom $R(d)$ of $\wp(U)^\UP$ such that $R(d) \subseteq R(m)$. 
By the minimality of $m$, we have $R(m) = R(d)$. Since $R(d)$ is a block, also $R(m)$ is  a block. 
Therefore, the set of minimal elements of $(U,\lesssim_R)$ is $\{ d \in U \mid \text{$R(d)$ is a block} \}$.

Let $x \in U$. Since  $R(x) \in \wp(U)^\UP$, there exists $R(d)$ such that $R(d)$ is a block and 
$R(d) \subseteq R(x)$. But $R(d) \subseteq R(x)$ means $d \lesssim_R x$ and since $d$ is minimal by the above, 
$U$ is bounded by the minimal elements.
\end{proof}

\begin{theorem} \label{Thm:Characterization}
For a tolerance $R$ on $U$, the following are equivalent:
\begin{enumerate}[\rm (a)]
 \item $R$ is induced by an irredundant covering.
 \item $(U,\lesssim_R)$ is bounded by minimal elements and $R$ equals ${\approx_R}$.
\end{enumerate}
\end{theorem}

\begin{proof}  (a)$\Rightarrow$(b): By Lemma~\ref{Lem:Minimal}, $(U,\lesssim_R)$ is bounded by minimal elements. Thus,
it suffices to show that $R$ equals ${\approx_R}$. If $a \, R \, b$, then by Lemma~\ref{Lem:IrredundantBlocks} and the equivalence of (D1) and (D2), 
there exists $d$ such that $R(d)$ is a block of $R$ and $a,b \in R(d)$. By  Lemma~\ref{Lem:ToleranceBlocks}, we have $R(d) \subseteq R(a) \cap R(b)$, and hence $a \approx_R b$.
Conversely, if  $a \approx_R b$, then by Corollary~\ref{Cor:Upset}, there exists a $\lesssim_R$-minimal element $d$ such that $a,b \in {\uparrow}d$,
that is, $R(d) \subseteq R(a),R(b)$. Notice that by Lemma~\ref{Lem:Minimal}, $d$ is an element such that $R(d)$ is a block. Then, $R(d) \subseteq R(a),R(b)$
implies  $a,b \in R(d)$, and since $R(d)$ is a block, we have $a \, R \, b$.

(b)$\Rightarrow$(a): If $(U,\lesssim_R)$ is bounded by minimal elements, then by Lemma~\ref{Lem:MinimalElements}(b),
the tolerance $R = {\approx_R}$ is induced by the irredundant covering  $\mathcal{H} = \{ {\uparrow} m \mid \text{$m$ is minimal\,} \}$.
\end{proof}

\begin{example} \label{Exa:Graph2}
For a given tolerance $R$ on $U$, the minimal elements $x$ in $(U,\lesssim_R)$ are such that $R(x)$ is a block, as 
is shown in the proof of Lemma~\ref{Lem:Minimal}. In Example~\ref{Exa:Graph1}, we noted that $R(x)$ is a block if and only if 
$R(x)$ is a clique.

Consider the graph $\mathcal{G}$ defined by the tolerance $R$ on $U = \{a,b,c,d\}$ depicted in Figure~\ref{Fig:figure1}.
Note that this graph originates from Example 1.3 in \cite{Pogonowski81}.
The $R$-neighbourhood of the elements $a$, $b$, and $d$ forms a clique. Thus, they are minimal in  $(U,\lesssim_R)$.
\begin{figure}[ht]
\centering
\includegraphics[width=100mm]{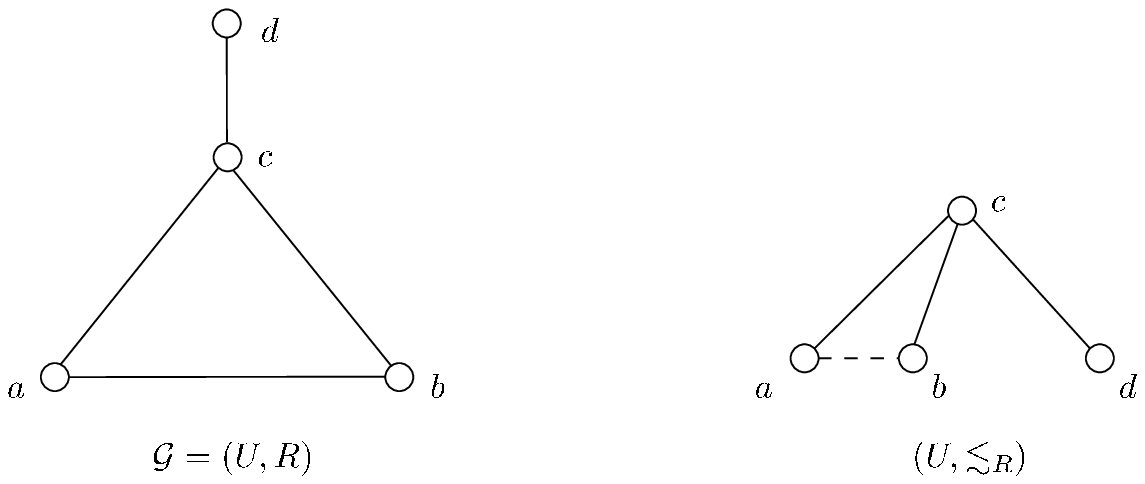}
\caption{\label{Fig:figure1}}
\end{figure}
Note that the elements $a$ and $b$ are such that $a \lesssim_R b$ and $b \lesssim_R a$, but $a \neq b$. This is expressed in the diagram
of $(U,\lesssim_R)$ by the dashed segment, and the solid lines depict the ``natural'' partial order between the elements. Because $U$ is bounded 
by minimal elements and $\approx_R$ is equal to $R$, the family $\{ \{a,b,c\}, \{c,d\}\}$ of the $R$-neighbourhoods of the minimal elements form an 
irredundant covering inducing $R$ by Theorem~\ref{Thm:Characterization}. 
\end{example}

\begin{remark} \label{Rem:SummariseCoverings}
Let $R$ be a tolerance induced by an irredundant covering $\mathcal{H}$. Since the sets in $\mathcal{H}$ are of the form ${\uparrow} m = R(m)$, 
where $m$ is a minimal element of $(U, \lesssim_R)$, we have that $x \, R \, y$ if and only if there exists a minimal element element $m$ in  
$(U, \lesssim_R)$ such that $x \, R \, m$ and $y \,R \, m$. Thus, the minimal elements can be considered as a ``model'' of the tolerance $R$, 
because $R$ is completely determined by the $R$-neighbourhoods of the minimal elements. 
In \cite{JarvRade2014}, we called these elements \emph{complete}. We also know that $\mathcal{H} \subseteq \mathcal{B}(R)$
and this inclusion can be proper. Thus, there can be such blocks of $R$ that are not determined by a complete element. In Section~\ref{Sec:Helly},
we will concentrate on characterising the case in which $\mathcal{H} = \mathcal{B}(R)$.
\end{remark}
 
We end this section considering shortly the notion of a canonical base introduced in \cite{Schreider75}. Let $R$ be a tolerance on $U$. A subset
$\mathcal{K} \subseteq \mathcal{B}(R)$ is called a \emph{canonical base} of $R$ if $\mathcal{K}$ induces $R$, but  for any $X \in \mathcal{K}$,
the family $\mathcal{K} \setminus \{ X \}$ does not induce $R$. 

If $R$ is induced by an irredundant covering of $\mathcal{H}$, then $\mathcal{H} \subseteq \mathcal{B}(R)$, as we noted already. In addition,
$\mathcal{H}$ consists of the upsets ${\uparrow} m$ of the minimal elements of $(U,\lesssim_R)$. This implies that $\mathcal{H}$ is
a canonical base, because if we remove ${\uparrow} m$ from $\mathcal{H}$, then the pair $(m,m)$ does not belong to the relation induced by
the family $\mathcal{H} \setminus \{ {\uparrow} m \}$, but trivially $(m,m) \in R$.

\begin{proposition} \label{Prop:UniqueBase}
Let $R$ be a tolerance induced by a family $\mathcal{H} \subseteq \mathcal{B}(R)$. Then, the following assertions are equivalent:
\begin{enumerate} [\rm (a)]
 \item $\mathcal{H}$ is an irredundant covering of $U$;
 \item $\mathcal{H}$ is the unique canonical base and $\mathcal{H} \subseteq \{ R(x) \mid x \in U\}$.
\end{enumerate}
\end{proposition}

\begin{proof} 
(a)$\Rightarrow$(b): If $\mathcal{H}$ is an irredundant covering of $U$, then $\mathcal{H} \subseteq \{ R(x) \mid x \in U\}$
and $\mathcal{H}$ is a canonical base.
Let $\mathcal{K} \subseteq \mathcal{B}(R)$ be an arbitrary canonical base of $R$. 
Then, $\mathcal{K}$ is also a covering of $U$, and hence for any minimal element
$m$ of $(U,\lesssim_R)$, there exists a block $B \in \mathcal{K}$ such that $m \in B$.
Because $B$ is a block, $B \subseteq R(m) \in \mathcal{H}$. Since $R(m)$ is also a block of $R$,
we get $R(m)=B$. Thus $\mathcal{H} \subseteq \mathcal{K}$, and we get $\mathcal{H=K}$,
because $\mathcal{K}$ is a canonical base.

(b)$\Rightarrow$(a): The reflexivity of $R$ implies that any canonical base is a covering of $U$. 
Because the covering $\mathcal{H}$ satisfies condition (D2), it is irredundant.
\end{proof}

\begin{example}
In Proposition~\ref{Prop:UniqueBase}, we showed that if $R$ is a tolerance induced by an irredundant covering,
then this covering is also the unique canonical base. In this example, we show that
unique bases do not necessarily form irredundant coverings.
Consider the tolerance $R$ on $\{a,b,c,d\}$ such that $R(a) = \{a,b\}$, $R(b) = \{a,b,c\}$,
$R(c) = \{b,c,d\}$, and $R(d) = \{c,d\}$. The family $\mathcal{B}(R) = \{ \{a,b\}, \{b,c\}, \{c,d\}\}$
of the blocks of $R$ is the unique canonical base of $R$, but  $\mathcal{B}(R)$ does
not form an irredundant covering of $U$, because the block $\{b,c\}$ can be removed and we still
have a covering of $U$.
\end{example}

\section{Helly number and irredundant coverings} \label{Sec:Helly}

The \emph{Helly number} of a family of sets $\mathcal{H}$ is $k$, if for any subfamily $\mathcal{S}$ of $\mathcal{H}$, any $k$ members of $\mathcal{S}$ 
have a common point, then the sets in $\mathcal{S}$ have a common point. This number is named after Edward Helly (1884--1943). 
Let us now define a similar notion in quasiordered sets. Let $(U,\lesssim)$ be a quasiordered set. We say that $(U,\lesssim)$ has
the \emph{Helly number} $k$, if for any subset $A \subseteq U$, any $k$ elements of $A$ have a common lower bound, then $A$ has a
common lower bound. In particular, the Helly number of $(U,\lesssim)$ is 2 if and only if every preblock of $\approx$  has a lower bound.

\begin{theorem} \label{Thm:HellyCharacterisation}
Let $(U,\lesssim)$ be a quasiordered set. Then, the following assertions are equivalent:
\begin{enumerate}[\rm (a)]
 \item The Helly number of $(U,\lesssim)$ is 2; 
 \item $\mathcal{B}(\approx) = \{ {\uparrow} m \mid \text{$m$ is minimal}\}$;
 \item $\mathcal{B}(\approx)$ is an irredundant covering and $(U,\lesssim)$ is bounded by minimal elements.
\end{enumerate}
\end{theorem}

\begin{proof} (a)$\Rightarrow$(b):  
Let $B$ be a block of $\approx$. Because $U$ has the Helly number 2, $B$ has a lower bound $m$. 
Thus, $B \subseteq {\uparrow}m$, and because $B$ is a block and ${\uparrow} m$ is a preblock, we have $B = {\uparrow} m$. 
This element $m$ is minimal, because if $a \lesssim m$, then ${\uparrow} m \subseteq {\uparrow} a$. Since ${\uparrow} m$ 
is a block and ${\uparrow} a$ is a preblock, we have  ${\uparrow} m = {\uparrow} a$. Thus, $a \in {\uparrow}m$ gives 
$m \lesssim a$ and $m$ is minimal. On the other hand, if $m$ is minimal, then ${\uparrow} m$ is a block by 
Lemma~\ref{Lem:MinimalElements}(a). Thus, $\mathcal{B}(\approx) = \{ {\uparrow} m \mid \text{$m$ is a minimal} \}$.

(b)$\Rightarrow$(c):
If $\mathcal{B}(\approx) = \{ {\uparrow} m \mid \text{$m$ is minimal}\}$, then for all
$x \in U$, there exists a minimal $m_x$ such that  $x \in {\uparrow} m_x$ and
$m_x \lesssim x$. Thus, $(U, \lesssim)$ is bounded by minimal elements. By Lemma~\ref{Lem:MinimalElements}(b), the
covering $\mathcal{B}(\approx)$ is irredundant.

(c)$\Rightarrow$(a):
Assume that $\mathcal{B}(\approx)$ is an irredundant covering and $(U,\lesssim)$ is bounded
by minimal elements. Then, by Lemma~\ref{Lem:MinimalElements}(b), $\mathcal{H} = \{ {\uparrow} m \mid \text{$m$ is minimal}\}$ 
is an irredundant covering inducing $\approx$. Lemma~\ref{Lem:IrredundantBlocks} implies $\mathcal{H} \subseteq \mathcal{B}(\approx)$. 
But since also $\mathcal{B}(\approx)$ is irredundant, we must have $\mathcal{H} = \mathcal{B}(\approx)$. 
Suppose $P$ is a $\approx$-preblock. Then, there exists a block $B$ of $\approx$ such that $P \subseteq B \in \mathcal{H}$.
Because  $\mathcal{H}$ consists of such upsets ${\uparrow} m$ that $m$ is minimal, there
exists a minimal $m$ such that $B = {\uparrow m}$. Thus, $m$ is a lower bound of $B$ and $P$, 
and the Helly number of $(U,\lesssim)$ is 2.
\end{proof}

In the following, we characterise for finite quasiordered sets the cases in which their Helly number is 2. If $(U,\lesssim)$
is a finite quasiordered set, then $(U,\lesssim)$ has at least one minimal element. This can be seen as follows.
Let $a_1$ be any element of $U$. If $a_1$ is not minimal, then there exists an element $a_2$ such
that $a_2 \lesssim a_1$, but $a_2 \not \gtrsim a_1$. If $a_2$ is minimal, then we are finished.  If $a_2$
is not minimal, then there is $a_3$ such that $a_3 \lesssim a_2$ and $a_3 \not \gtrsim a_2$. 
If $a_3$ is not minimal, continue this procedure. Since $A$ is finite, this process must terminate.

\begin{proposition} \label{Prop:FiniteHelly}
Let $(U,\lesssim)$ be a finite quasiordered set. The following are equivalent:
\begin{enumerate}[\rm (a)]
 \item The Helly number of  $(U,\lesssim)$ is 2.
 \item If for some minimal elements $a_1,a_2,a_3 \in U$, the intersections ${\uparrow} a_1 \cap {\uparrow}a_2$, 
 ${\uparrow} a_1 \cap {\uparrow}a_3$,  ${\uparrow} a_2 \cap {\uparrow}a_3$ are nonempty, then there exists
 a minimal element $a \in U$ such that
 \[ ({\uparrow} a_1 \cap {\uparrow}a_2) \cup ({\uparrow} a_1 \cap {\uparrow}a_3) \cup ({\uparrow} a_2 \cap {\uparrow}a_3)
  \subseteq {\uparrow} a.
 \]
\end{enumerate} 
\end{proposition}

\begin{proof}
(a)$\Rightarrow$(b): 
Let $a_1$, $a_2$, $a_3$ be minimal elements and 
let $H =  ({\uparrow} a_1 \cap {\uparrow}a_2) \cup ({\uparrow} a_1 \cap {\uparrow}a_3) \cup ({\uparrow} a_2 \cap {\uparrow}a_3)$.
If $x,y \in H$, then there exists $a_i$ for some $1 \leq i \leq 3$ such that $x \gtrsim a_i \lesssim y$ and $x \approx y$. This means
that $H$ is a preblock, and since (a) holds, $H$ has a lower bound $b$. Because $(U,\lesssim)$ is finite, there exists a minimal element
$a$ such that $a \lesssim b$. Hence, $H \subseteq {\uparrow} a$.

(b)$\Rightarrow$(a): We show by induction that each preblock $X$ of $\approx$ has a lower bound. Let $X$ be a preblock.
If $|X| = 1$, then the only element in $X$ is trivially a lower bound of $X$. If $|X| = 2$, then $X = \{x,y\}$
for some $x,y \in U$, and $x \approx y$ means that there is $a \in U$ such that $a \lesssim x,y$. Thus, $a$ is a lower
bound for $X$.  As an induction hypothesis, we assume that each preblock $X$  of $\approx$  such that $|X| \leq n$ has a lower bound.
Let us consider a preblock $X$ such that $|X| = n + 1$. Since $X$ has now at least three elements,
$X$ can be partitioned into three disjoint nonempty sets $X_1$, $X_2$, $X_3$. It is obvious that the sets $X_1 \cup X_2$, 
$X_1 \cup X_3$, $X_2 \cup X_3$ are preblocks having at most $n$ elements. By the induction hypothesis, these preblocks have
lower bounds. Since $(U,\lesssim)$ is finite, there exists minimal elements $a_1, a_2, a_3 \in U$ such that $a_1 \lesssim x_1$
for all $x_1 \in X_1 \cup X_2$, $a_2 \lesssim x_2$ for all $x_2 \in X_1 \cup X_3$, and $a_3 \lesssim x_3$
for all $x_3 \in X_2 \cup X_3$. This implies that
\[ X_1 \subseteq ({\uparrow} a_1 \cap {\uparrow}a_2),\  X_2 \subseteq ({\uparrow} a_1 \cap {\uparrow}a_3), \ 
  X_3 \subseteq ({\uparrow} a_2 \cap {\uparrow}a_3)
\]
Because the intersections ${\uparrow} a_1 \cap {\uparrow}a_2$, ${\uparrow} a_1 \cap {\uparrow}a_3$,  
${\uparrow} a_2 \cap {\uparrow}a_3$ are nonempty and (b) holds, there exists a minimal element $a$ such that
\[ X = X_1 \cup X_2 \cup X_3 \subseteq ({\uparrow} a_1 \cap {\uparrow}a_2) \cup ({\uparrow} a_1 \cap {\uparrow}a_3) 
 \cup ({\uparrow} a_2 \cap {\uparrow}a_3) \subseteq {\uparrow} a.
\]
Thus, the preblock $X$ has a lower bound $a$.
\end{proof}

We will call a covering $\mathcal{H}$ of a set $U$ \emph{normal} if it satisfies the conditions (ii) and (iv) from 
page 19 of the book \cite{Chajda91}, that is, 

\begin{enumerate}[({N}1)]
\item $\mathcal{H}$ is an antichain with respect to $\subseteq$.

\item If $M \subseteq U$ is a set such that $M \nsubseteq B$ for all
$B \in \mathcal{H}$, then there exists at least one two-element set $\{a,b\}\subseteq M$
such that $\{a,b\} \nsubseteq B$ for all $B \in \mathcal{H}$.
\end{enumerate}

Theorem 2.9 of \cite{Chajda91} (see also \cite{Chajda76, Czedli82, Czedli83, Gratzer89}) states that a covering $\mathcal{H}$ of $U$ is normal if and only if 
there is a tolerance $R$ on $U$ such that $\mathcal{H} = \mathcal{B}(R)$. For a normal covering $\mathcal{H}$, this $R$ is 
just the tolerance induced by $\mathcal{H}$. Thus, there is a one-to-one correspondence
between  tolerances on $U$ and normal coverings on $U$. Note also that in \cite{Zelinka77},
normal coverings, and their relationship to cliques of a graph, are studied under the name $\tau$-coverings. 
As an immediate consequence of \cite{Chajda91} and \cite{Zelinka77}, we obtain:

\begin{remark} \label{Rem:Normal}
Let $\mathcal{H}$ be a normal covering of $U$. If for some $B_{1}, B_{2}, B_{3} \in \mathcal{H}$, 
the intersections $B_{1}\cap B_{2}$, $B_{1}\cap B_{3}$, and $B_{2} \cap B_{3}$ are nonempty, then
there exists $B \in \mathcal{H}$ that includes these intersections. Indeed, 
let $C = (B_{1} \cap B_{2}) \cup (B_{1} \cap B_{3}) \cup (B_{2} \cap B_{3})$. Then, for any
$\{a,b\} \subseteq C$, we have $\{a,b\} \subseteq B_{i}$ for some $i\in\{1,2,3\}$. 
Since $\mathcal{H}$ is a normal covering of $U$, $C \nsubseteq B$ for all $B \in \mathcal{H}$ is not possible. 
Therefore, there exists $B\in \mathcal{H}$ that includes $C$.
\end{remark}

In Lemma~\ref{Lem:IrredundantBlocks}, we showed that if $R$ is a tolerance induced by an irredundant covering $\mathcal{H}$, 
then $\mathcal{H} \subseteq \mathcal{B}(R)$. The natural question then is, when such a $\mathcal{H}$ coincides with
$\mathcal{B}(R)$? An obvious answer is that exactly in the case the covering $\mathcal{H}$ is normal. Next we present a theorem
that connects normal coverings, Helly numbers of 2, and the condition given in Remark~\ref{Rem:Normal} for
tolerances induced by an irredundant finite covering. For that,
we need to introduce some notation. Let $R$ be a tolerance on $U$. We denote by
\[ \mathcal{N} = \{ R(x) \mid x \in U\} \]
the set of all $R$-neighbourhoods of the elements of $U$. This set can be ordered by the set-inclusion $\subseteq$. Thus,
$(\mathcal{N},\subseteq)$ is a partially ordered set. Observe that 
$x \lesssim_R y$ \ in $U$ if and only if $R(x) \subseteq R(y)$ in $\mathcal{N}$. The following conditions are now obvious:
\begin{itemize}
 \item The element $m \in U$ is a minimal element of $(U,\lesssim_R)$ if and only if $R(m)$ is a minimal element of $(\mathcal{N},\subseteq)$.
 \item The Helly number of $(U,\lesssim_R)$ is 2 if and only if the Helly number of $(\mathcal{N},\subseteq)$ is 2.
\end{itemize}

If $R$ is a tolerance induced by a finite irredundant covering $\mathcal{H}$ of $U$, then $\mathcal{N}$ is
finite. Indeed, 
\[ \mathcal{H} = \{ R(m) \mid \text{$m$ is minimal} \} = \{ R(m) \mid \text{$R(m)$ is an atom of $\wp(U)^\UP$} \} \]
and hence $\wp(U)^\UP$ is finite, because it is generated by a finite set $\mathcal{H}$ of atoms. Therefore, $\mathcal{N} \subseteq \wp(U)^\UP$
must be finite.

\begin{proposition}\label{Prop:Main}
Let $R$ be a tolerance induced by a finite irredundant covering $\mathcal{H}$ of $U$. Then, the following assertions are equivalent:
\begin{enumerate}[\rm (i)]
\item $\mathcal{H} = \mathcal{B}(R)$.

\item $\mathcal{H}$ is a normal covering.

\item If for some $B_{1}, B_{2}, B_{3} \in \mathcal{H}$, the intersections $B_{1} \cap B_{2}$, $B_{1} \cap B_{3}$, 
and $B_{2} \cap B_{3}$ are nonempty, then there exists $B \in \mathcal{H}$ that includes these intersections.

\item The Helly number of $(U,\lesssim _R)$ is 2.
\end{enumerate}
\end{proposition}

\begin{proof} The equivalence (i)$\Leftrightarrow$(ii) is an immediate consequence
of \cite[Theorem~2.9]{Chajda91} as we already noted, and the implication (ii)$\Rightarrow$(iii)
is clear by Remark~\ref{Rem:Normal}.

(iii)$\Rightarrow$(iv): Let $R(a_{1})$, $R(a_{2})$, and $R(a_{3})$ be minimal
elements in $\mathcal{N}$ such that the intersections $R(a_{1})\cap R(a_{2})$,
$R(a_{1})\cap R(a_{3})$, and $R(a_{2})\cap R(a_{3})$ are nonempty. Then, the
elements $a_{1}$, $a_{2}$, and $a_{3}$ are minimal in $(U,\lesssim_{R})$ and
${\uparrow} a_{1}$, ${\uparrow} a_{2}$, and ${\uparrow} a_{3}$ belong to $\mathcal{H}$,
since $\mathcal{H}$ is an irredundant covering. We have that the
intersections ${\uparrow} a_{1} \cap {\uparrow} a_{2}$, ${\uparrow} a_{1} \cap {\uparrow} a_{3}$, 
and ${\uparrow} a_{2}\cap {\uparrow} a_{3}$ are nonempty, which by
assumption (iii) implies that there exists a minimal $a$ in $(U,\lesssim_{R})$
such that ${\uparrow} a \in \mathcal{H}$ includes the intersections ${\uparrow} a_{1} \cap {\uparrow} a_{2}$, 
${\uparrow} a_{1} \cap {\uparrow} a_{3}$, and ${\uparrow} a_{2} \cap {\uparrow} a_{3}$. Therefore, 
by Proposition~\ref{Prop:FiniteHelly}, the Helly number of $(U,\lesssim_{R})$ is two.

\noindent(iv)$\Rightarrow$(i). Assume that (iv) holds. Then, by Theorem~\ref{Thm:HellyCharacterisation},
the blocks of the tolerance $\approx_{R}$ form an irredundant covering of
$U$ and $\mathcal{B}(\approx_{R}) = \{ {\uparrow} m \mid m \text{ is minimal}\}$.
Since by assumption, $R$ is induced by an irredundant covering $\mathcal{H}$, 
Theorem~\ref{Thm:Characterization} implies that $R$ equals $\approx_{R}$. 
By Lemma~\ref{Lem:IrredundantBlocks}, we have $\mathcal{H} \subseteq \mathcal{B}(R) = \mathcal{B}(\approx_R)$.
Because both $\mathcal{H}$ and $\mathcal{B}(\approx_R)$ are irredundant coverings inducing $R$,
we get $\mathcal{H} = \mathcal{B}(R)$.
\end{proof}

\begin{example}\label{Exa:Helly}
Let us consider the tolerance on $U = \{a,b,c,d,e,f,g\}$ given in Figure~\ref{Fig:figure2}. The corresponding quasiordered set
$(U, \lesssim_R)$ is bounded by minimal elements $a$, $c$, $g$. In addition, the tolerance $R$ coincides with the relation
$\approx_R$ and thus $R$ is induced by an irredundant covering $\mathcal{H} = \{ {\uparrow} a, {\uparrow} c, {\uparrow} g\}$.

\begin{figure}[h]
\centering
\includegraphics[width=120mm]{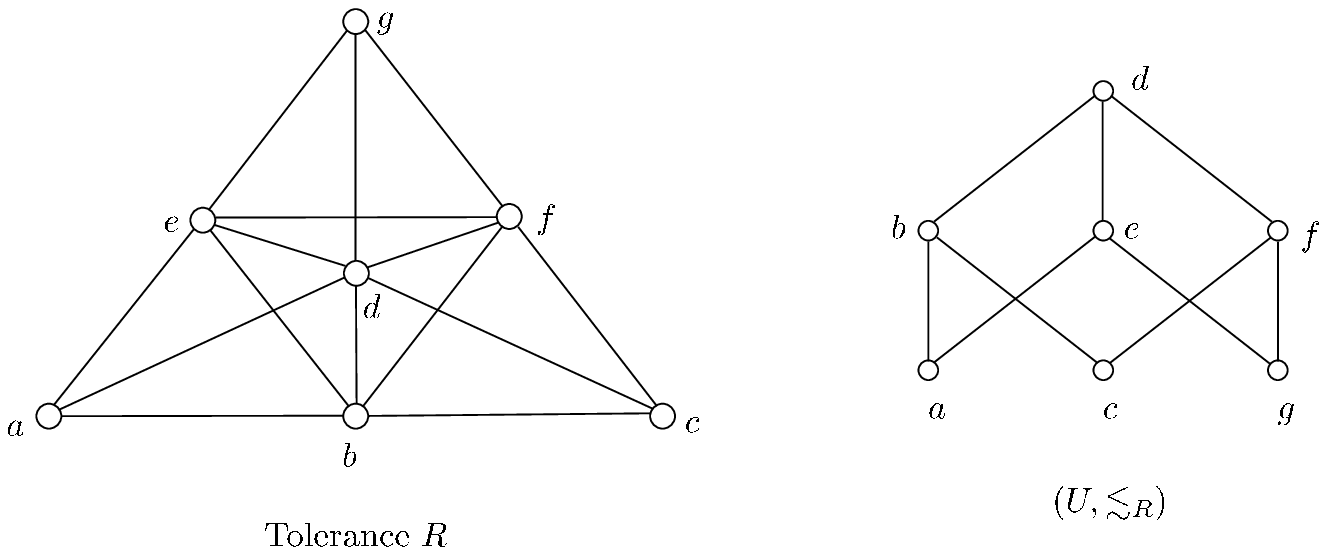}
\caption{\label{Fig:figure2}}
\end{figure}
We have that the following intersections of elements in $\mathcal{H}$ are nonempty:
\[ 
{\uparrow} a \cap {\uparrow} c = \{b,d\},  {\uparrow} a \cap {\uparrow} g = \{d,e\}, \text{ and } {\uparrow} c \cap {\uparrow} g = \{d,f\}.
\]
Obviously, there is no $B \in \mathcal{H}$ that includes the union $\{b,d,e,f\}$ of the above intersections. By Proposition~\ref{Prop:Main},
we have $\mathcal{H} \subset \mathcal{B}(R)$. In fact, $\{b,d,e,f\}$ is the only block of $R$ not belonging to $\mathcal{H}$.
\end{example}

It is also interesting to notice that the quasiordered set $(U,\lesssim_R)$ of Figure~\ref{Fig:figure2} is almost isomorphic to the
Boolean lattice $\mathbf{2^3}$ except that the least element is missing. Next we consider the tolerances induced by finite distributive
lattices. Let $\mathcal{L} = (L,\leq)$ be a lattice with a least element $0$. We denote the set $L\setminus\{0\}$ by $L^+$. We say that the
Helly number of the lattice $\mathcal{L}$ is $k$, if the partially ordered set ($L^{+},\leq)$ has the Helly number $k$.

\begin{corollary} \label{Cor:DistributiveLattice}
If $\mathcal{L} = (L,\leq)$ is a finite distributive lattice, then the Helly number of $\mathcal{L}$ is 
$2$ if and only if $\mathcal{L}$ has at most two different atoms.
\end{corollary}

\begin{proof} Because $\mathcal{L}$ is finite, it has a least element $0$,
and its atoms are precisely the minimal elements of $(L^{+},\leq)$.
If their number is one or two, then condition (b) of 
Proposition~\ref{Prop:FiniteHelly} is trivially
satisfied, and hence the Helly number of $(L^{+},\leq)$ and $\mathcal{L}$ is $2$.

Conversely, suppose that the Helly number of $\mathcal{L}$ is $2$.
Let us first recall the well-known facts that for any $x, y \in L$, 
${\uparrow} x \cap {\uparrow} y = {\uparrow} (x \vee y)$,  
and $x \geq y$ if and only if ${\uparrow} x \subseteq {\uparrow} y$. 

Let $a_{1},a_{2},a_{3}\in L$ be three distinct atoms of $\mathcal{L}$. 
First notice that they are minimal elements in $(L^{+},\leq)$ and that the upsets corresponding to these atoms can not 
be disjoint, because each of them contains the greatest element $1$ of the finite lattice $\mathcal{L}$.
By applying Proposition~\ref{Prop:FiniteHelly}, we get that there exists an atom $a\in L$ such that
${\uparrow} (a_{1} \vee a_{2}) \subseteq {\uparrow} a$,  
${\uparrow} (a_{1} \vee a_{3}) \subseteq {\uparrow} a$, and
${\uparrow} (a_{2} \vee a_{3}) \subseteq {\uparrow} a$.
Since $L$ is distributive and $a_{1},a_{2},a_{3}$ are atoms, this implies
 \[ a \leq (a_{1} \vee a_{2}) \wedge(a_{1} \vee a_{3}) \wedge (a_{2} \vee a_{3}) = 
    (a_{1} \wedge a_{2}) \vee(a_{1} \wedge a_{3}) \vee(a_{2}\wedge a_{3}) = 0,
 \]
a contradiction.
\end{proof}

We may now conclude that each finite distributive lattice $\mathcal{L} = (L,\leq)$ induces a tolerance $({\geq} \circ {\leq})$  
on $L^+$. Let us denote this tolerance by $\bowtie$. Because $L^+$ is finite, it is
trivially bounded by minimal elements. By Lemma~\ref{Lem:MinimalElements}, $\mathcal{H} = \{ {\uparrow} m \mid \text{$m$ is minimal} \}$
is an irredundant covering of $L^+$, and the tolerance induced by $\mathcal{H}$ is $\bowtie$. By Corollary~\ref{Cor:DistributiveLattice},
the Helly number of $\mathcal{L}$ is $2$ if and only if $\mathcal{L}$ has at most two different atoms. Finally, we obtain that
$\mathcal{B}(\bowtie) = \mathcal{H}$ if and only if $\mathcal{L}$ has at most two different atoms by Proposition~\ref{Prop:Main}.

\section*{Acknowledgements}
We thank the anonymous referees for their valuable remarks on our manuscript.

 %\bib, bibdiv, biblist are defined by the amsrefs package.
\bibliographystyle{fundam}

\end{document}